\documentclass[12pt,oneside]{amsart}
\usepackage{amsmath,amssymb,cite,mathrsfs,tikz-cd}
\usepackage[all]{xy}
\usepackage{typearea} 
\usepackage{hyperref} 
\usepackage{color} 

\newcounter{maincounter}
\numberwithin{maincounter}{section}
\numberwithin{equation}{section}

\newtheorem{lemma}[maincounter]{Lemma}
\newtheorem{proposition}[maincounter]{Proposition}

\newtheorem{remark}[maincounter]{Remark}
\newtheorem{theorem}[maincounter]{Theorem}
\newtheorem{conjecture}[maincounter]{Conjecture}

\def\AA{\mathbb{A}}

\def\NN{\mathbb{N}}
\def\RR{\mathbb{R}}

\def\CC{\mathbb{C}}

\def\PP{\mathbb{P}}
\def\QQ{\mathbb{Q}}

\newcommand{\cal}{\mathcal}

\newcommand{\cA}{\cal{A}}

\newcommand{\IP}{{\PP}}

\newcommand{\IC}{{\CC}}
\newcommand{\IR}{{\RR}}

\newcommand{\IQbar}{{\overline{\QQ}}}

\newcommand{\IN}{{\NN}}
\newcommand{\IA}{{\AA}}
\newcommand{\IQ}{{\QQ}}

\newcommand{\cL}{{\mathcal L}}
\newcommand{\cM}{{\mathcal M}}
\newcommand{\cO}{{\mathcal O}}

\newcommand{\ssm}{\setminus}

\newcommand{\codim}{{\rm codim}}

\renewcommand{\subset}{\subseteq} 

\newcommand{\jac}[1]{\mathrm{Jac}({#1})}

\newcommand{\pullbackcorner}[1][dr]{\save*!/#1-1.7pc/#1:(-1.5,1.5)@^{|-}\restore}

\DeclareMathOperator{\rk}{rk}

\begin{document}
\title{A consequence of the Relative Bogomolov Conjecture}
\author{Vesselin Dimitrov}
\author{Ziyang Gao}
\author{Philipp Habegger}

\address{Department of Pure Mathematics and Mathematical Statistics, Centre for Mathematical Sciences, Wilberforce Road, Cambridge CB3 0WA, UK}
\email{vesselin.dimitrov@gmail.com}
\address{CNRS, IMJ-PRG, 4 place de Jussieu, 75005 Paris, France}
\email{ziyang.gao@imj-prg.fr}
\address{Department of Mathematics and Computer Science, University of Basel, Spiegelgasse 1, 4051 Basel, Switzerland}
\email{philipp.habegger@unibas.ch}

\subjclass[2000]{11G30, 11G50, 14G05, 14G25}

\maketitle
\begin{abstract} We propose a formulation of the relative Bogomolov conjecture and show that it gives an affirmative answer to a question of Mazur's concerning the uniformity of the Mordell--Lang conjecture for curves. In particular we show that the relative Bogomolov conjecture implies the uniform Manin--Mumford conjecture for curves. The proof is built up on our previous work \cite{DGHUnifML}.
\end{abstract}
\tableofcontents

\section{Introduction}

Let $F$ be a field of characteristic $0$. A smooth curve $C$ defined
over $F$ is a geometrically irreducible, smooth, projective curve
defined over $F$. We denote by $\mathrm{Jac}(C)$ the Jacobian of $C$.
The following conjecture is a question posed by Mazur \cite[top of page~234]{mazur1986arithmetic}.

\begin{conjecture}\label{ConjMazur}
Let $g \ge 2$ be an integer. Then there exists a constant $c(g) \ge 1$ with the following property. Let $C$ be a smooth curve of genus $g$ defined over $F$, let $P_0 \in C(F)$, and let $\Gamma$ be a subgroup of $\mathrm{Jac}(C)(F)$ of finite rank $\rk(\Gamma)$. Then
\begin{equation}\label{EqBoundMazur}
\#(C(F)-P_0)\cap\Gamma \le c(g)^{1+\rk(\Gamma)}
\end{equation}
where $C - P_0$ is viewed as a curve in $\mathrm{Jac}(C)$ via the Abel--Jacobi map based at $P_0$.
\end{conjecture}

When $F = \IQbar$, based on Vojta's method, R\'{e}mond \cite{Remond:Decompte} has proved an explicit upper bound of $\#(C(\IQbar)-P_0)\cap\Gamma$. Apart from $g$ and $\mathrm{rk}(\Gamma)$, R\'{e}mond's bound depends also on a suitable height of $\mathrm{Jac}(C)$.

Two particular consequences of Conjecture~\ref{ConjMazur} are:
\begin{enumerate}
\item[(i)] Take $F$ a number field and $\Gamma = \mathrm{Jac}(C)(F)$. By the Mordell--Weil Theorem $\mathrm{Jac}(C)(F)$ is a finitely generated abelian group. Then \eqref{EqBoundMazur} becomes a bound on the number of rational points $\#C(F) \le c(g)^{1+\rk\mathrm{Jac}(C)(F)}$.
\item[(ii)] If $F = \IC$ and $\Gamma = \mathrm{Jac}(C)_{\mathrm{tor}}$, then \eqref{EqBoundMazur} becomes $\#(C(\IC)-P_0)\cap  \mathrm{Jac}(C)_{\mathrm{tor}} \le c(g)$, the uniform Manin--Mumford conjecture for curves in their Jacobians.
\end{enumerate}

In a recent work, we proved \eqref{EqBoundMazur} provided that the modular height of the curve in question is large in terms of $g$; see  \cite[Theorem~1.2]{DGHUnifML}. Prior to our work and for torsion points, \textit{i.e.} $F = \IC$ and $\Gamma = \mathrm{Jac}(C)_{\mathrm{tor}}$, the desired bound \eqref{EqBoundMazur} was proved by DeMarco--Krieger--Ye \cite{DeMarcoKriegerYeUniManinMumford} for any genus $2$ curve admitting a degree-two map to an elliptic curve when the Abel--Jacobi map is based at a Weierstrass point.

\vskip 0.5em

The goal of this note is to give a precise statement for the folklore \textit{relative Bogomolov conjecture}, and prove that it implies the full Conjecture~\ref{ConjMazur}.

\subsection{The Relative Bogomolov Conjecture}
We start by proposing a formulation for the relative Bogomolov conjecture. 

Let $S$ be a regular, irreducible, quasi-projective variety. Let $\pi \colon \cA \rightarrow S$ be an abelian scheme
of relative dimension $g \ge 1$, namely a proper smooth group scheme whose fibers are abelian varieties. Let $\cL$ be a symmetric relatively
ample line bundle on $\cA/S$. Assume that $S$, $\pi$ and $\cL$ are all defined over $\IQbar$. 

For each $s \in S(\IQbar)$, the line bundle $\cL_s$ on the abelian
variety $\cA_s = \pi^{-1}(s)$ is symmetric and ample; note that $\cA_s$ is defined over $\IQbar$. Tate's Limit Process provides a fiberwise N\'{e}ron--Tate height $\hat{h}_{\cA_s,\cL_s} \colon \cA_s(\IQbar) \rightarrow [0,\infty)$; it vanishes precisely on the torsion points in $\cA_s(\IQbar)$. Finally define $\hat{h}_{\cL} \colon \cA(\IQbar) \rightarrow [0,\infty)$ to be $P \mapsto \hat{h}_{\cA_{\pi(P)},\cL_{\pi(P)}}(P)$.

Let $\eta$ be the generic point of $S$ and fix an algebraic closure of
the function field of $S$. For any subvariety $X$ of
$\cA$ that
dominates $S$, denote by $X_{\overline{\eta}}$ the geometric generic
fiber of $X$.
In particular, $\cA_{\overline{\eta}}$ is an abelian variety  over an
algebraically closed field.

\begin{conjecture}[Relative Bogomlov Conjecture]
  \label{ConjRelBog}
  Let $X$ be an irreducible subvariety of $\cA$ defined over
  $\overline{\IQ}$ that dominates $S$. Assume that
  $X_{\overline{\eta}}$ is irreducible
  and not contained in any proper
  algebraic subgroup of $\cA_{\overline{\eta}}$. If $\codim_{\cA} X >
  \dim S$, then there exists $\epsilon > 0$ such that
  \[ X(\epsilon; \cL) := \{ x \in X(\IQbar) : \hat{h}_{\cL}(x) \le
    \epsilon \}
  \]
  is not Zariski dense in $X$.
\end{conjecture}

The name \textit{Relative Bogomolov Conjecture} is reasonable: the
same statement with $\epsilon = 0$ is precisely the relative
Manin--Mumford conjecture proposed by Pink~\cite[Conjecture 6.2]{Pink}
and  Zannier~\cite{ZannierBook}, which is proved when
$\dim X=1$ in a series of papers \cite{MasserZannierTorsionPointOnSqEC, MASSER2014116, MasserZannierRelMMSimpleSur, CorvajaMasserZannier2018, MasserZannierRMMoverCurve}.

The classical Bogomolov conjecture, proved by Ullmo \cite{Ullmo} and
S.~Zhang \cite{ZhangEquidist}, is precisely
Conjecture~\ref{ConjRelBog} for $\dim S = 0$. When $\dim S = 1$ and
$X$ is the image of a section, Conjecture~\ref{ConjRelBog} is
equivalent to  S.~Zhang's conjecture in his 1998 ICM note
\cite[$\mathsection$4]{zhang1998small} if $\cA_{\overline{\eta}}$ is
simple and is proved by DeMarco--Mavraki
\cite[Theorem~1.4]{DeMarcoMavraki} if $\cA \rightarrow S$ is isogenous
to a fiber product of elliptic surfaces. In general Conjecture~\ref{ConjRelBog} is still open.

After this paper is completed, K\"{u}hne \cite{KuehneRBC} proved the Relative Bogomolov Conjecture for arbitrary subvarieties of a fiber product of elliptic surfaces.

\subsection{Main result}
Our main result, which is built up on \cite{DGHUnifML}, is the following theorem.

\begin{theorem}\label{MainThm}
 Assume that the Relative Bogomolov Conjecture, \textit{i.e.}, Conjecture~\ref{ConjRelBog}, holds true. Then Conjecture~\ref{ConjMazur} holds true.
\end{theorem}

The proof of Theorem~\ref{MainThm} is as follows. First we reduce
Conjecture~\ref{ConjMazur} to the case $F = \IQbar$ by using a
specialization  result of Masser~\cite{masser1989specializations}. This is executed in $\mathsection$\ref{SectionSpecialization}. When $F = \IQbar$, our proof follows closely and uses our previous work \cite{DGHUnifML}, where we proved \eqref{EqBoundMazur} for all curves whose modular height is bounded below by a constant depending only on $g$; see \cite[Theorem~1.2]{DGHUnifML}. The key point in \cite{DGHUnifML} is to prove a height inequality, which we cite as Theorem~\ref{ThmHeightInequalityDGH} in the current paper. As is shown by the proof of \cite[Proposition~7.1]{DGHUnifML}, the extra hypothesis on the modular height of curves required in \cite[Theorem~1.2]{DGHUnifML} is necessary because of the constant term in this height inequality. 
In this paper, we show that this constant term can be removed if we assume the Relative Bogomolov Conjecture; see Proposition~\ref{PropHeightInequalitystrengthend} for a precise statement. Then by following the framework presented in \cite{DGHUnifML} we can prove our Theorem~\ref{MainThm} for $F = \IQbar$.

\medskip

We emphasize that what the Relative Bogomolov Conjecture does is to \textit{complement} \cite{DGHUnifML} to prove the full  Conjecture~\ref{ConjMazur}. More precisely, \cite[Theorem~1.2]{DGHUnifML} proves Conjecture~\ref{ConjMazur} for curves $C$ whose modular height is larger than a number $\delta = \delta(g)$ depending only on the genus $g$, and the Relative Bogomolov Conjecture can handle curves with small modular height.

After this paper is completed, K\"{u}hne proved the \textit{Uniform Bogomolov Conjecture} for curves \cite{Kuehne:21} which also handles curves with small modular height. He can replace the \textit{Relative Bogomolov Conjecture} in the proof of Theorem~\ref{MainThm} and thus obtain an \textit{unconditional} proof of Conjecture~\ref{ConjMazur}. We refer to the survey \cite{GaoSurveyUML} for more detailed comments.

\subsection*{Acknowledgements} We would like to thank the referee for their comments. Vesselin Dimitrov has received funding
from the European Union's Seventh Framework Programme (FP7/2007--2013)
/ ERC grant agreement n$^\circ$ 617129. Ziyang Gao has received
fundings from the French National Research Agency grant
ANR-19-ERC7-0004, and the European Research Council (ERC) under the
European Union's Horizon 2020 research and innovation programme (grant
agreement n$^\circ$ 945714). Philipp Habegger has received funding
from the Swiss National Science Foundation (grant n$^\circ$ 200020\_184623).

\section{Proof of the main result for $F=\IQbar$}
\label{sec:Qbarcase}
In this section we prove Theorem~\ref{MainThm} when $F = \IQbar$.
\begin{theorem}
  \label{thm:mainthmQbarcase}
  The Relative Bogomolov Conjecture, \textit{i.e.}, Conjecture \ref{ConjRelBog},
  implies Conjecture \ref{ConjMazur} for $F=\IQbar$. 
\end{theorem}

The proof follows closely and is based on our previous work \cite{DGHUnifML}. 

\subsection{Basic setup}
\label{sec:basicsetup}
Fix an integer $g \ge 2$.
Let $\mathbb{M}_g$ be the fine moduli space of smooth projective curves of genus $g$ 
with level-$4$-structure, \textit{cf.}
 \cite[Chapter~XVI, Theorem~2.11 (or above
Proposition~2.8)]{ACG:Curve},
 \cite[(5.14)]{DM:irreducibility}, or \cite[Theorem
 1.8]{OortSteenbrink}.  It is known that
$\mathbb{M}_g$ is an irreducible regular quasi-projective variety
defined over $\IQbar$, and $\dim \mathbb{M}_g = 3g-3$.
This variety solves the underlying moduli problem.
There exists a universal curve $\mathfrak{C}_g$ over
$\mathbb{M}_g$, it is smooth and proper over
$\mathbb{M}_g$ with fibers that are smooth curves of genus
$g$. Moreover, it is equipped with level
$3$-structure. 

Let $\mathrm{Jac}(\mathfrak{C}_g)$ be the relative Jacobian of
$\mathfrak{C}_g \rightarrow \mathbb{M}_g$. It is an abelian scheme
equipped with a natural principal polarization and with level-$3$-structure; see
\cite[Proposition~6.9]{MFK:GIT94}.

Let  $\mathbb{A}_g$ be the fine moduli space of principally polarized
abelian varieties of dimension $g$ with level-$3$-structure. It is known that
$\mathbb{A}_g$ is an irreducible regular quasi-projective variety defined over $\IQbar$; see
\cite[Theorem~7.9 and below]{MFK:GIT94} or \cite[Theorem
1.9]{OortSteenbrink}.
Here too we have a universal object, the universal abelian scheme
$\pi \colon \mathfrak{A}_g \rightarrow \mathbb{A}_g$ of fiber
dimension $g$.
There exists a relatively very ample line bundle
$\cL$ on $\mathfrak{A}_g/\mathbb{A}_g$
satisfying $[-1]^*\cL = \cL$; see \cite[Th\'{e}or\`{e}me~XI~1.4]{LNM119}. By \cite[Proposition~4.4.10(ii) and Proposition~4.1.4]{EGAII}, we then have a closed immersion $\mathfrak{A}_g \rightarrow \IP^n_{\IQbar}\times \mathbb{A}_g$ over $\mathbb{A}_g$ arising from $\cL\otimes\pi^*\cM^{\otimes p}$, where $\cM$ is an ample line bundle on $\mathbb{A}_g$, for some integer $p \ge 1$.

Attaching the Jacobian to a smooth curve induces the Torelli morphism
$\tau\colon \mathbb{M}_g \rightarrow \mathbb{A}_g$. The famous Torelli
theorem states that, absent level structure, the Torelli morphism is
injective on $\IC$-points.
In out setting,  $\tau$ is finite-to-$1$ on $\IC$-points, \textit{cf.} \cite[Lemma~1.11]{OortSteenbrink}. 
As $\mathbb{A}_g$ is a fine moduli space we have the following Cartesian diagram
\begin{equation}
\label{EqUnivJac}
\xymatrix{
\mathrm{Jac}(\mathfrak{C}_g) \ar[r] \ar[d]  \pullbackcorner & \mathfrak{A}_g \ar[d]^{\pi} \\
\mathbb{M}_g \ar[r]_{\tau} & \mathbb{A}_g
}
\end{equation}

Suppose we have an immersion $\mathbb{A}_g \subseteq \IP_{\IQbar}^N$
defined over $\IQbar$, such an immersion exists. We write
$\overline{\mathbb{A}_g}$ for the Zariski closure of $\mathbb{A}_g$ in
$\IP_{\IQbar}^N$. Then the absolute logarithmic Weil height on
$\IP^N_{\IQbar}(\IQbar)$ restricts to a height function
$h_{\overline{\mathbb{A}_g}} \colon \overline{\mathbb{A}_g}(\IQbar)
\rightarrow \IR$. Thus $h_{\overline{\mathbb{A}_g}}$ represents the
Weil height attached to the ample line bundle obtained by restricting
$\cO(1)$ on $\IP_{\IQbar}^N$ to $\overline{\mathbb{A}_g}$. Moreover,
$h_{\overline{\mathbb{A}_g}}$ takes values in $[0,\infty)$ as the
Weil height function is non-negative. 


For $M\in\IN=\{1,2,3,\ldots\}$ we write $\mathfrak{A}_g^{[M]}$ for the
$M$-fold fibered power
$\mathfrak{A}_g\times_{\mathbb{A}_g} \cdots \times_{\mathbb{A}_g}\mathfrak{A}_g$
over $\mathbb{A}_g$.
Then $\mathfrak{A}_g^{[M]}\rightarrow \mathbb{A}_g$ is an abelian
scheme.

Similarly, for  any morphism of schemes $S \rightarrow \mathbb{M}_g$,
the base change is
$\mathfrak{C}_S = \mathfrak{C}_g \times_{\mathbb{M}_g} S$. 
Furthermore, the $M$-fold fibered power  $\mathfrak{C}_S \times_S \cdots \times_S
\mathfrak{C}_S$ is denoted by $\mathfrak{C}_S^{[M]}$.
The
morphism $\mathfrak{C}_S\rightarrow S$ is smooth and therefore open,
thus so is $\mathfrak{C}_S^{[M]}\rightarrow S$.
Each fiber of the latter morphism is a product of curves and thus
irreducible. We conclude that $\mathfrak{C}_S^{[M]}$ is irreducible if
$S$ is.

Suppose $C$ is a smooth curve defined over a field and
$A = \jac{C}$. 
The  difference morphism  $C^{M+1}\rightarrow A^M$ determined by
\begin{equation}
  \label{eq:fiberwisefz}
  (P_0,P_1,\ldots,P_M) \mapsto (P_1-P_0,\ldots,P_M-P_0).
\end{equation}
is well-defined; we do not need to specify a base point for the
Abel--Jacobi map. It is an
 astonishingly powerful tool in diophantine geometry.

Let us make this more precise in our relative setting.
For each morphism of schemes $S \rightarrow \mathbb{M}_g
\xrightarrow{\tau} \mathbb{A}_g$, we briefly recall the  construction
of the proper  $S$-morphism
\begin{equation}
\label{eq:faltingszhang0}
\mathscr{D}_M\colon \mathfrak{C}_S^{[M+1]}\rightarrow \mathfrak{A}_g^{[M]} \times_{\mathbb{A}_g} S
\end{equation}
from \cite[$\mathsection$6.1]{DGHUnifML}.
 Indeed by the proof of \cite[Proposition~6.9]{MFK:GIT94} there is a
 morphism $\iota\colon \mathfrak{C}_S \rightarrow
 \mathrm{Pic}^1(\mathfrak{C}_S/S)$ to the line bundles of degree $1$.
 Let $f_1,f_2$ be any two morphisms from an $S$-scheme $T$
 with target $\mathfrak{C}_S$. Then
 the difference of $\iota\circ f_1$ and $\iota\circ f_2$ is a
  morphism
  $T\rightarrow \mathrm{Pic}^0(\mathfrak{C}_S/S) = \mathrm{Jac}(\mathfrak{C}_S)$.
So we get a morphism $\mathfrak{C}_S^{[2]}\rightarrow
\mathrm{Jac}(\mathfrak{C}_S)$. We compose with the natural morphism
$\mathrm{Jac}(\mathfrak{C}_S)\rightarrow \mathfrak{A}_g\times_{\IA_g} S$ coming from
the Torelli morphism. This construction extends to $M+1$ section and
yields (\ref{eq:faltingszhang0}).
Fiberwise  the morphism $\mathscr{D}_M$ behaves on points as (\ref{eq:fiberwisefz}).


The morphism $\mathscr{D}_M$ in \eqref{eq:faltingszhang0} is called
the \textit{$M$-th Faltings--Zhang} map. Note that if $S$ is
irreducible, then $\mathscr{D}_M(\mathfrak{C}_S^{[M+1]})$ is an
irreducible subvariety of the abelian scheme $\mathfrak{A}_g^{[M]}
\times_{\mathbb{A}_g} S \rightarrow S$.

We will use the following theorem, which we proved in \cite{DGHUnifML}, by applying \cite[Theorem~1.6]{DGHUnifML} to \cite[Theorem~1.2']{GaoBettiRank}.
\begin{theorem}\label{ThmHeightInequalityDGH}
Let $S$ be an irreducible variety with a (not necessarily dominant) quasi-finite morphism $S \rightarrow \mathbb{M}_g$. Assume $g \ge 2$ and $M \ge 3g-2$. Then there exist constants $c > 0$ and $c' \ge 0$ and a Zariski open dense subset $U$ of $\mathscr{D}_M(\mathfrak{C}_S^{[M+1]})$ with
\begin{equation}\label{EqHeightInequalityDGH}
\hat{h}_{\cL}(P) \ge c h_{\overline{\mathbb{A}_g}}(\pi(P)) - c' \quad\text{for all}\quad P \in U(\IQbar).
\end{equation}
\end{theorem}

\subsection{A strengthend height inequality} 
We use the notation in the previous subsection. Here we prove that the Relative Bogomolov Conjecture allows us to furthermore strengthen the height inequality given by Theorem~\ref{ThmHeightInequalityDGH}. Let $g \ge 2$. 
\begin{proposition}\label{PropHeightInequalitystrengthend}
Let $S$ be an irreducible variety with a (not necessarily dominant) quasi-finite morphism $S \rightarrow \mathbb{M}_g$. Let $M$ be an integer satisfying $M \ge 3g-1$ if $g=2$ and $M \ge 3g-2$ if $g\ge 3$. 

Assume that the Relative Bogomolov Conjecture, Conjecture~\ref{ConjRelBog},
holds true. 
Then there exist a constant $c > 0$ and a Zariski open dense subset $U$ of $\mathscr{D}_M(\mathfrak{C}_S^{[M+1]})$ with
\begin{equation}\label{EqHeightInequalityDGH2}
\hat{h}_{\cL}(P) \ge c \max\{1,h_{\overline{\mathbb{A}_g}}(\pi(P))\} \quad\text{for all}\quad P \in U(\IQbar).
\end{equation}
\end{proposition}
\begin{proof}
  The fiber
  of $\mathfrak{C}_S^{[M+1]}\rightarrow S$ above $s\in S(\IQbar)$,
  is a product of $M+1$ curves of genus $g$. The Faltings--Zhang map
  is generically finite on this product. So 
  we have
  \begin{equation}
    \label{eq:codimbound}
    \codim_{\mathfrak{A}_g^{[M]} \times_{\mathbb{A}_g} S} \mathscr{D}_M(\mathfrak{C}_S^{[M+1]}) = Mg - (M+1) = (g-1)M -1 > 3g-3 \ge \dim S.
  \end{equation}
  
Let $\overline{\eta}$ be the geometric generic point of $S$.
Each fiber of
  $\mathscr{D}_M(\mathfrak{C}_S^{[M+1]})\rightarrow S$ is the image of
the $(M+1)$-fold power of a smooth curve under the Faltings--Zhang
map. So all fibers are irreducible and thus so is the geometric
generic fiber $\mathscr{D}_M(\mathfrak{C}_S^{[M+1]})_{\overline{\eta}}$.
As a smooth curve generates its Jacobian we find that
$\mathscr{D}_M(\mathfrak{C}_S^{[M+1]})_{\overline{\eta}}$ is not
contained in a  proper algebraic subgroup of $\mathfrak{A}_g^{[M]}
\times_{\mathbb{A}_g} \overline{\eta}$. By (\ref{eq:codimbound}) the image  $X :=
\mathscr{D}_M(\mathfrak{C}_S^{[M+1]})$ satisfies the assumptions of
the
Relative Bogomolov Conjecture. Thus there exists $\epsilon >
0$ such that $\overline{X(\epsilon ; \cL)}^{\mathrm{Zar}}$,
the Zariski
closure of 
$X(\epsilon ; \cL) = \{x \in X(\IQbar) : \hat{h}_{\cL}(x) \le
\epsilon\}$,
is not equal to $X$. 

Let $c >0$, $c' \ge 0,$ and $U$ be as in Theorem~\ref{ThmHeightInequalityDGH}. The paragraph above implies that 
\[
U \setminus \overline{X(\epsilon ; \cL)}^{\mathrm{Zar}}
\]
is still Zariski open and dense in $X$.


It suffices to prove that \eqref{EqHeightInequalityDGH2} holds true
with $c$ a positive constant that is independent of $P$ and with
$U$ replaced $U \setminus \overline{X(\epsilon ; \cL)}^{\mathrm{Zar}}$. 

Take any $P \in (U \setminus \overline{X(\epsilon ;
  \cL)}^{\mathrm{Zar}})(\IQbar)$. So
$\hat h_\cL(P) \ge \epsilon$ and
$\hat h_\cL(P) \ge c h_{\overline{\IA_g}}(\pi(P))-c'$. 

We split up into the two cases depending on whether
$\max\{1,h_{\overline{\mathbb{A}_g}}(\pi(P))\}\le \max\{1,2c'/c\}$ holds
or does not hold.

In the first case we have 
$$
\hat h_{\cL}(P)\ge \frac{\epsilon}{\max\{1,2c'/c\}} 
\max\{1,h_{\overline{\mathbb{A}_g}}(\pi(P))\}
$$
and \eqref{EqHeightInequalityDGH2} follows with $\epsilon /
\max\{1,2c'/c\}$ for the constant $c$.

In the second case we have
$h_{\overline{\mathbb{A}_g}}(\pi(P))>
\max\{1,2c'/c\}$ and hence
$c h_{\overline{\mathbb{A}_g}}(\pi(P))-c'\ge c
h_{\overline{\mathbb{A}_g}}(\pi(P))/2$. Thus
$$
\hat h_{\cL}(P)\ge \frac c2
\max\{1,h_{\overline{\mathbb{A}_g}}(\pi(P))\}. 
$$
Again \eqref{EqHeightInequalityDGH2} holds with $c/2$ for the constant
$c$. 
%
\end{proof}
\begin{remark}
The proof of Proposition~\ref{PropHeightInequalitystrengthend} is the only place where the Relative Bogomolov Conjecture is used in the proof of Theorem~\ref{MainThm}.
\end{remark}

\subsection{Dichotomy on the N\'{e}ron--Tate distance between points
  on curves}
We use the notation of Subsection~\ref{sec:basicsetup}.

\begin{proposition}\label{PropNTDistance}
Let $S$ be an irreducible closed subvariety of $\mathbb{M}_g$.

Assume that the Relative Bogomolov Conjecture,
Conjecture~\ref{ConjRelBog}, holds true. 
Then there exist positive constants $c_2, c_3, c_4$ depending on $S$
with the following property. For all $s \in S(\IQbar)$
 there is a subset $\Xi_s \subseteq \mathfrak{C}_s(\IQbar)$ with
 $\#\Xi_s \le c_2$ such that any $P \in \mathfrak{C}_s(\IQbar)$
 satisfies one of the following cases.
\begin{enumerate}
\item[(i)] We have $P \in \Xi_s$;
\item[(ii)] or $\#\bigl\{Q \in \mathfrak{C}_s(\IQbar) : \hat{h}_{\cL}(Q-P) \le c_3^{-1}\max\{1,h_{\overline{\mathbb{A}_g}}(\tau(s))\}\bigr\} < c_4$.
\end{enumerate}
\end{proposition}
We start the numbering of the constants from $c_2$ to make the
statement comparable to \cite[Proposition~7.1]{DGHUnifML}. The proof
of this proposition is similar to the proof of
\cite[Proposition~7.1]{DGHUnifML}. The main difference is that the
height inequality Theorem~\ref{ThmHeightInequalityDGH} is replaced by the strengthend version Proposition~\ref{PropHeightInequalitystrengthend}, and this allows us to remove the extra hypothesis on $h_{\overline{\mathbb{A}_g}}(\tau(s))$ in \cite[Proposition~7.1]{DGHUnifML}.

\begin{proof} We prove the proposition by induction on $\dim S$.
  The induction start
  $\dim S =0$ is treated as part of the induction step.


We fix $M$ as in
Proposition~\ref{PropHeightInequalitystrengthend}. Then there exist a
constant $c > 0$ and a Zariski open dense subset $U$ of
$X=\mathscr{D}_M(\mathfrak{C}_{S}^{[M+1]})$ satisfying the following
property. For all $s \in S(\IQbar)$ and all $P, Q_1,\ldots,Q_M \in
\mathfrak{C}_s(\IQbar)$, we have

\begin{equation}\label{EqHtInequalityBdRatPt}
  c \max\{1,h_{\overline{\mathbb{A}_g}}(\tau(s))\} \le
  \hat{h}(Q_1-P) + \cdots + \hat{h}(Q_M - P)  
  \text{ if } (Q_1-P,\ldots, Q_M-P) \in U(\IQbar).
\end{equation}

Observe that $\pi_S(X)=S$, where $\pi_S\colon
\mathfrak{A}_{g}^{[M]}\times_{\mathbb{M}_g} S\rightarrow S$ is the
structure morphism. Thus $S \setminus \pi_S(U)$ is not Zariski dense in
$S$. Let $S_1,\ldots, S_r$ be the irreducible components of the
Zariski closure of $S\setminus \pi_S(U)$ in $S$. Then $\dim S_j \le
\dim S - 1$ for all $j$.

Note that if $\dim S = 0$, then $\pi_S(U)=S$ and $r=0$.
If $\dim S\ge 1$, then
this  proposition holds for all, if any,
$S_j$ by the induction hypothesis.
So it
remains to prove the conclusion of this proposition for curves above
\begin{equation}
\label{eq:soutsideSi}
s \in S(\IQbar) \setminus \bigcup_{j=1}^r S_j(\IQbar).
\end{equation}

The image of $s$ under the Torelli map is $\tau(s)$. For any $P \in
\mathfrak{C}_s(\IQbar)$, we will consider $\mathfrak{C}_s - P$ as a
curve inside $(\mathfrak{A}_g)_{\tau(s)}$ via the Abel--Jacobi map
based at $P$.

The set $W = X\setminus U$ is a Zariski closed proper  subset of
$X$. The fiber  $W_s$ of $W$ above $s$ satisfies $W_s\subsetneq X_s =
\mathscr{D}_M(\mathfrak{C}_s^{[M+1]})$
as  (\ref{eq:soutsideSi}) implies $s\in \pi_S(U(\IQbar))$. 

We define
\begin{equation}
  \label{def:Xis}
  \Xi_s = \{P \in \mathfrak{C}_s(\IQbar) : (\mathfrak{C}_s-P)^M
  \subseteq W_s\}. 
\end{equation}
Note that $\Xi_s = \bigcup_{Z} \Xi_Z$ where $Z$ ranges over the 
 irreducible components   of $W_s$ and
$\Xi_Z := \{ P \in \mathfrak{C}_s(\IQbar) : ( \mathfrak{C}_s- P)^M
\subseteq Z \}$. 
We can thus apply  \cite[Lemma~6.4]{DGHUnifML} to $A = (\mathfrak{A}_g)_{\tau(s)},
C=\mathfrak{C}_s-P\subset A$, and each $Z$. As $Z\subsetneq
\mathscr{D}_M(\mathfrak{C}_s^{[M+1]})$ we have $\#\Xi_Z \le 84(g-1)$.
The precise value of $84(g-1)$, the classical upper bound for the automorphism
group of a curve of genus $g$, is not so important; but its uniformity
in $s$ is. 
As $W_s$ appears as a fiber of $W$ considered as
a family over $S$, we see that
number of irreducible components of $W_s$ is bounded from above by a
constant $c'_2$ that is independent of $s$.
Our estimates imply
$\#\Xi_s \le c_2$ where  $c_2 =
84(g-1)c'_2$ is also independent of $s$.

We have constructed $\Xi_s$ which serves as the exceptional
set in part
(i). We will assume that (i) fails, \textit{i.e.}, 
$P\in \mathfrak{C}_s(\IQbar)\ssm \Xi_s$ and conclude (ii).
So $(\mathfrak{C}_s-P)^M
\not\subseteq W_s$ by (\ref{def:Xis}). Our next task is
to apply \cite[Lemma~6.3]{DGHUnifML} to
$\mathfrak{C}_s-P$ and $W_s$.

Recall that $\mathfrak{A}_g$ can be embedded in $\IP^n_{\IQbar} \times
\mathbb{A}_g$; see $\mathsection$\ref{sec:basicsetup} above \eqref{EqUnivJac}.
The image of $s\in \mathbb{M}_g(\IQbar)$ under the Torelli morphism is  
$\tau(s)\in\mathbb{A}_g(\IQbar)$. 
We may take $\mathfrak{C}_s-P$ as a smooth curve in $\IP_{\IQbar}^n$. 
The degree of $\mathfrak{C}_s$ as a subvariety of 
$\IP_\IQbar^n$ is bounded from above independently 
of $s$. As translating inside an abelian variety
does not affect the degree, we see that the degree of
$\mathfrak{C}_s-P$ is bounded from above independently of $s$; see \cite[Lemma~6.1(i)]{DGHUnifML}.
Recall that $W_s$ is a Zariski closed subset of
$X_s\subset (\mathfrak{A}_g^{[M]})_{\tau(s)}$
we may identify it with a Zariski closed subset
of $(\IP_{\IQbar}^n)^M$.
The degree
of $W_s$ is bounded from above independently of $s$
as it is the fiber  above $\tau(s)$ of a
subvariety of $(\IP_{\IQbar}^n)^M\times \mathbb{A}_g$.
From \cite[Lemma~6.3]{DGHUnifML} we
thus obtain a number $c_4$, depending only  on these bounds but
not on $s$  with the
following property. Any subset $\Sigma \subset
\mathfrak{C}_s(\IQbar)$ with cardinality $\ge c_4$ satisfies $(\Sigma-P)^M
\not\subset W_s(\IQbar)$.

Finally set $\Sigma = \bigl\{Q \in \mathfrak{C}_s(\IQbar) : \hat{h}(Q-P)
\le  c_3^{-1} \max\{1,h_{\overline{\mathbb{A}_g}}(\tau(s))\} \bigr\}$ with $c_3 = 2M/c$. 

It remains to prove $\#\Sigma < c_4$, in which case we are in case (ii) of the proposition and hence we are done. Suppose $\#\Sigma \ge c_4$. Then $(\Sigma-P)^M \not\subset W_s(\IQbar)$. So
 there exist $Q_1,\ldots,Q_M \in
\Sigma$ such that $(Q_1 - P, \ldots, Q_M-P) \in
U(\IQbar)$. Thus we can apply \eqref{EqHtInequalityBdRatPt} and obtain
\begin{equation*}
  c\max\{1,h_{\overline{\mathbb{A}_g}}(\tau(s))\} \le 
  M\frac{c}{2M}
      \max\{1,h_{\overline{\mathbb{A}_g}}(\tau(s))\}  =
  \frac{c}{2} 
\max\{1,h_{\overline{\mathbb{A}_g}}(\tau(s))\},
\end{equation*}
a contradiction. 
\end{proof}


\subsection{Completion of the proof of Theorem~\ref{MainThm} for $F=\IQbar$}
We follow the argumentation in~\cite{DGH1p}, or more precisely
\cite[$\mathsection$8]{DGHUnifML}. We will assume that
Conjecture~\ref{ConjRelBog} holds true.

Let $C$ be a smooth genus $g \ge 2$ defined over
$\IQbar$, and let $\Gamma$ be a subgroup of $\mathrm{Jac}(C)(\IQbar)$
of finite rank $\rho$. Let $P_0 \in C(\IQbar)$.

The curve $C$ corresponds to a $\IQbar$-point $s_{\mathrm{c}}$ of
$\mathbb{M}_{g,1}$, the coarse moduli space of smooth curves of genus
$g$ without level structure.

The fine moduli space $\mathbb{M}_g$ of smooth curves of genus $g$
with level-$4$-structure admits  a finite and surjective morphism of
$\mathbb{M}_{g,1}$. So there exists an $s \in \mathbb{M}_g(\IQbar)$
that maps to $s_{\mathrm{c}}$. Thus $C$ is isomorphic, over $\IQbar$,
to the fiber $\mathfrak{C}_s$ of the universal curve $\mathfrak{C}_g
\rightarrow \mathbb{M}_g$ above $s$. We thus view $\Gamma$ as a finite rank subgroup of $\mathrm{Jac}(\mathfrak{C}_s)(\IQbar)$, and $P_0 \in \mathfrak{C}_s(\IQbar)$.

A standard application of R\'{e}mond's explicit formulation of the Vojta and Mumford inequalities~\cite{remond:vojtasup,Remond:Decompte} yields the following bound. There exists a constant $c = c(g) \ge 1$ with the following property. There is $P_s \in \mathfrak{C}_s(\IQbar)$ such that
\[
\#\left\{P \in \mathfrak{C}_s(\IQbar) : P-P_s \in \Gamma, ~ \hat{h}_{\cL}(P-P_s) > c \max\{1, h_{\overline{\mathbb{A}_g}}(\tau(s))\} \right\} \le c^{\rho};
\]
we refer to Lemma 8.2 and the proof of Proposition 8.1, both~\cite{DGHUnifML}, for details.

Let us start by conditionally verifying
Conjecture~\ref{ConjMazur}
for $P_0 = P_s$. 
In this case it suffices to prove
\begin{equation}\label{EqSmallPoints}
\#\left\{P \in \mathfrak{C}_s(\IQbar) : P-P_s \in \Gamma, ~ \hat{h}_{\cL}(P - P_s) \le c \max\{1, h_{\overline{\mathbb{A}_g}}(\tau(s))\} \right\} \le {c'}^{1+\rho}
\end{equation}
for some $c'$ independent of $s$. To do this we apply Proposition~\ref{PropNTDistance} to $S=\mathbb{M}_g$. Let $c_2, c_3, c_4$ be from this proposition; they are independent of $s$. If $P\in \mathfrak{C}_s(\IQbar)$, then either $P \in \Xi_s$ for some $\Xi_s \subseteq \mathfrak{C}_s(\IQbar)$ with $\#\Xi_s \le c_2$ or 
\begin{equation}\label{Eq2ndAlt}
\#\left\{Q \in \mathfrak{C}_s(\IQbar) : \hat{h}_{\cL}(Q-P) \le c_3^{-1}\max\{1,h_{\overline{\mathbb{A}_g}}(\tau(s))\}\right\} < c_4.
\end{equation}
So to prove the desired bound \eqref{EqSmallPoints} we may assume $P \in \mathfrak{C}_s(\IQbar) \setminus \Xi_s$ and thus \eqref{Eq2ndAlt}.

Now we apply the ball packing argument as in \cite{DGHUnifML} to prove the bound \eqref{EqSmallPoints}. Consider
\begin{equation}\label{EqSmallPoints2}
\#\left\{P-P_s \in \Gamma :  P \in \mathfrak{C}_s(\IQbar)\setminus \Xi_s , ~ \hat{h}_{\cL}(P-P_s) \le c \max\{1, h_{\overline{\mathbb{A}_g}}(\tau(s))\} \right\}.
\end{equation}

Set $R = (c \max\{1, h_{\overline{\mathbb{A}_g}}(\tau(s))\})^{1/2}$.
We start by doing ball packing in the $\rho$-dimensional
$\IR$-vector space $\Gamma\otimes\IR$. It is well-known that $\hat
h^{1/2}$ defines an Euclidean norm on $\Gamma\otimes\IR$. The image in
$\Gamma\otimes\IR$ of the set in (\ref{EqSmallPoints2}) is contained in
the closed ball of radius $R$ centered at the image of $P_s$. Let
$r\in (0,R]$. By \cite[Lemme~6.1]{Remond:Decompte} a subset of
$\Gamma\otimes\IR$ that is contained in a closed ball of radius $R$ is
covered by at most $(1+2R/r)^{\rho}$ closed balls of radius $r$
centered at elements of the given set. The bound in (\ref{Eq2ndAlt})
suggests the choice $r = (c_3^{-1} \max\{1,
h_{\overline{\mathbb{A}_g}}(\tau(s))\})^{1/2}$. By possibly increasing
$c_3$ we may assume that the quotient $R/r = (c
c_3)^{1/2}$ lies in $[1,\infty)$. The crucial observation is that
$R/r$ 
 is independent of $s$.  So we
can  cover the image of the set in
(\ref{EqSmallPoints2}) in $\Gamma\otimes\IR$ with
 at most $c_5^{\rho}$ balls of radius $r$ where $c_5\ge 1$ is
independent of $s$.
  
 Say $P-P_s \in \Gamma$ maps to the center of a ball of radius $r$ in the covering. If $Q-P_s \in \Gamma$ maps to the same ball, then $\hat{h}(Q-P) \le r^2$. As $P \not \in \Xi_s$, by \eqref{Eq2ndAlt} the number of the $Q-P_s$'s that 
map to this closed ball of radius $r$  is
less than $c_4$. 
Thus the number of points in \eqref{EqSmallPoints2} is at most
$c_4 c_5^\rho$. So the number of points in the set from \eqref{EqSmallPoints} is at most $\#\Xi_s + c_4 c_5^\rho \le c_2 + c_4 c_5^\rho$, which is at most ${c'}^{1+\rho}$ for a suitable $c'$.
  This completes the proof of the proposition in the case $P_0=P_s$.
  
  Now we turn to a general $P_0\in \mathfrak{C}_s(\IQbar)$. The
  subgroup $\Gamma'$ of
$\mathrm{Jac}(\mathfrak{C}_s)(\IQbar)$ generated by $\Gamma$ and
$P_0-P_s$ has rank most $\rho +1$. Now if $Q\in
\mathfrak{C}_s(\IQbar)-P_0$ lies in $\Gamma$, then $Q+P_0-P_s\in
\mathfrak{C}_s(\IQbar)-P_s$ lies in $\Gamma'$.
We have just proved that the number of such $Q$ 
is at most $c^{1+\rk(\Gamma')}\le c^{2+\rho}\le (c^2)^{1+\rho}$ for a
$c \ge 1$ that is independent of $s$. \qed

\section{From $\IQbar$ to an arbitrary base field in characteristic $0$}\label{SectionSpecialization}
In this section we perform a specialization argument to reduce Conjecture~\ref{ConjMazur} to the case $F = \IQbar$. 
\begin{lemma}
  \label{lem:mazurspecialization}
  If Conjecture~\ref{ConjMazur} holds true for $F = \IQbar$, then it
  holds true for an arbitrary field $F$ of characteristic $0$.
\end{lemma}
\begin{proof}
Without loss of generality we may and do assume that $F = \overline{F}$.

Let $C$, $P_0 \in C(F)$, and $\Gamma$ be as in
Conjecture~\ref{ConjMazur} of rank $\rho$. By the definition of a finite rank group,
there exists a finitely generated subgroup $\Gamma_0$ of
$\mathrm{Jac}(C)(F)$ with rank $\rho$  such that
\[
\Gamma \subset \{x \in \mathrm{Jac}(C)(F) : [n]x \in \Gamma_0\text{ for some }n \in \IN\}.
\]

For each $n \in \IN$, define
\[
\frac{1}{n}\Gamma_0 := \{ x \in \mathrm{Jac}(C)(F) : [n]x \in \Gamma_0\}.
\]
Then $\frac{1}{n}\Gamma_0$ is again a finitely generated subgroup of $\mathrm{Jac}(C)(F)$ of rank $\rho$. Note that  $\{\frac{1}{n}\Gamma_0\}_{n \in \IN}$ is a filtered system and 
$\Gamma \subset \bigcup_{n\in \IN}\frac{1}{n}\Gamma_0$. So in order to
prove the desired bound \eqref{EqBoundMazur}, it suffices to prove
that  there exists a constant $c = c(g) > 0$ such that
\begin{equation}
  \label{EqBoundLeveln}
  \#(C(F)-P_0) \cap \frac{1}{n}\Gamma_0 \le c^{1+\rho}
\end{equation}
for each $n \in \IN$.

Let $n\in\IN$ and let $\gamma_1,\ldots,\gamma_{r} \in
\mathrm{Jac}(C)(F)$ be generators of $\frac{1}{n}\Gamma_0$
such that $\gamma_{\rho+1},\ldots,\gamma_{r}$ are torsion points, we
allow $r$ to depend on $n$.
There exists a field $K_n$, finitely generated over $\IQbar$, such that
$C$, $P_0$, and
the $\gamma_1,\ldots,\gamma_{r}$ are
defined over $K_n$. Then $K_n$ is the function field of
some regular, irreducible quasi-projective
variety $V_n$ defined over $\IQbar$.


Up to replacing $V_n$ by a Zariski open dense subset, $C$ extends to a
 smooth family $\mathfrak{C} \rightarrow V_n$ (\textit{i.e.}, $C$ is the generic
fiber of $\mathfrak{C} \rightarrow V_n$) with each fiber being a smooth
curve of genus $g$, the point $P_0$ extends to a section of
$\mathfrak{C} \rightarrow V_n$, and
$\gamma_1,\ldots,\gamma_{r}$ extend to
sections of the relative Jacobian
$\mathrm{Jac}(\mathfrak{C}/V_n) \rightarrow V_n$.
We retain the symbols $P_0,\gamma_1,\ldots,\gamma_{r}$ for these
sections.

Let $v\in V_n(\IQbar)$.
We may specialize $\gamma_1,\ldots,\gamma_{r}$
at $v$ and obtain
 elements $\gamma_1(v),\ldots, \gamma_{r}(v)$ of the fiber
 $\mathrm{Jac}(\mathfrak{C}/V_n)_v$ above $v$.

We define the specialization of $\frac{1}{n}\Gamma_0$ at $v$, which we
denote with $(\frac{1}{n}\Gamma_0)_v$, to be the subgroup of
$\mathrm{Jac}(\mathfrak{C}/V_n)_v(\IQbar) =
\mathrm{Jac}(\mathfrak{C}_v)(\IQbar)$ 
generated by
$\gamma_1(v),\ldots,\gamma_{r}(v)$. Note that
$\rk(\frac{1}{n}\Gamma_0)_v\le \rho$.

Suppose $\dim V_n\ge 1$, 
by \cite[Main Theorem and Scholium~1]{masser1989specializations}  there exists 
$v \in V_n(\IQbar)$
such that the specialization homomorphism $\frac{1}{n}\Gamma_0 \rightarrow
(\frac{1}{n}\Gamma_0)_v$  is injective.
If $\dim V_n=0$, then $V_n$ is a point $\{v\}$ and $K_n=\IQbar$. Here
specialization is the identity and the same conclusion holds.
 Thus if we denote by
$\mathfrak{C}_v - P_0(v)$ the curve in
$\mathrm{Jac}(\mathfrak{C}_v) = \mathrm{Jac}(\mathfrak{C}/V_n)_v$
obtained via the Abel--Jacobi map based at $P_0(v)$, the
specialization of $P_0$ at $v$, then we have
\[
\#(C-P_0)(F) \cap \frac{1}{n}\Gamma_0 \le \#\big(\mathfrak{C}_v - P_0(v) \big)(\IQbar) \cap \left(\frac{1}{n}\Gamma_0 \right)_v.
\]

By hypothesis, Conjecture~\ref{ConjMazur} holds true for $F = \IQbar$.
So the right-hand side of the
inequality above has an upper bound $c^{1+\rho}$ for some $c = c(g)
\ge 1$. Observe that this bound is independent of $n$. Thus we have established \eqref{EqBoundLeveln}. 
\end{proof}

\begin{proof}[Proof of Theorem~\ref{MainThm}]
  By hypothesis and Theorem~\ref{thm:mainthmQbarcase} Conjecture~\ref{ConjMazur} holds for
  $F=\IQbar$. So it suffices to apply Lemma~\ref{lem:mazurspecialization}.
\end{proof}

\section{Relative Bogomolov for isotrivial abelian schemes}
In this section, we prove that the relative Bogomolov conjecture holds
true for isotrivial abelian schemes as a consequence of S.~Zhang's
Theorem \cite{ZhangEquidist}. An abelian scheme $\cA\rightarrow S$ defined over $\IQbar$ is said to be isotrivial if
  there exists a finite and surjective morphism $S'\rightarrow S$ with
  $S'$ irreducible  such
  that $\cA\times_S S'$ is isomorphic to $A\times S'$ with $A$ an
  abelian variety defined over $\IQbar$.

Let $\cA \rightarrow S$ and $\cL$ be as above
Conjecture~\ref{ConjRelBog}.

\begin{proposition}
Conjecture~\ref{ConjRelBog} holds true if $\cA \rightarrow S$ is isotrivial.
\end{proposition}

\begin{proof}
  Let $X$ be an irreducible subvariety of $\cA$ defined over $\IQbar$
  that dominates $S$ such that $X_{\overline{\eta}}$ is irreducible
  and  not contained in any proper algebraic subgroup of $\cA_{\overline{\eta}}$; here $X_{\overline{\eta}}$ means the geometric generic fiber of $X$ and $\cA_{\overline{\eta}}$ means the geometric generic fiber of $\cA$. Assume
\begin{equation}\label{EqRBGIsotrivial}
\codim_{\cA}X > \dim S.
\end{equation}

\noindent\boxed{\text{Case: Trivial abelian scheme}} We start by proving the proposition when $\cA \rightarrow S$ is a trivial abelian scheme, \textit{i.e.}, $\cA = A \times S$ for some abelian variety $A$ over $\IQbar$.

Denote by $p \colon \cA = A \times S \rightarrow A$ the natural
projection. Let $L$ be an ample and symmetric line bundle on $A$ defined over
$\IQbar$. For simplicity  denote by $Y =
\overline{p(X)}^{\mathrm{Zar}}$. Then $\dim Y \le \dim X$, 
and $Y$ is not contained in any proper algebraic subgroup of $A$ by
our assumption on $X$.

Assume that for all $\epsilon > 0$, the set
\[
X(\epsilon; \cL) = \{x \in X(\IQbar) : \hat{h}_{\cL}(x) \le \epsilon \}
\]
is Zariski dense in $X$. We will prove a contradiction to \eqref{EqRBGIsotrivial}.

Since $\cL$ is relatively ample on $\cA \rightarrow S$, there exists
an integer $N \ge 1$ such that $\cL^{\otimes N} \otimes
(p^*L)^{\otimes -1}$ is relatively ample on $\cA \rightarrow S$. For
the N\'eron--Tate height functions on $\cA(\IQbar)$ we have $N\hat{h}_{\cL} \ge \hat{h}_{p^*L}$.

Take any $\epsilon > 0$. 
If $x \in X(\epsilon; \cL)$, then $\hat{h}_{\cL}(x) \le \epsilon$, and hence
\[
\hat{h}_L(p(x)) = \hat{h}_{p^*L}(x) \le N\hat{h}_{\cL}(x) \le N\epsilon.
\]
Letting $x$ run over elements in $X(\epsilon; \cL)$,  we then obtain
\begin{equation}\label{EqSmallPointsInclusion}
p(X(\epsilon ; \cL)) \subseteq Y(N\epsilon;L) := \{ y \in Y(\IQbar) : \hat{h}_L(y) \le N\epsilon\}.
\end{equation}

We have assumed $\overline{X(\epsilon ; \cL)}^{\mathrm{Zar}} = X$.
Applying $p$ to both sides and taking the Zariski closure, we get $\overline{p(X(\epsilon ;
  \cL))}^{\mathrm{Zar}} = Y$. Hence
$\overline{Y(N\epsilon;L)}^{\mathrm{Zar}} = Y$ by
\eqref{EqSmallPointsInclusion}. Recall that $Y$ is not contained in
any proper algebraic subgroup of $A$. As $N\epsilon$ runs over all positive real numbers, the classicial Bogomolov conjecture, proved by S.~Zhang \cite{ZhangEquidist}, implies that $Y = A$.

So $\dim X \ge  \dim Y = \dim A = \dim \cA - \dim S$, and thus $\codim_{\cA}X \le \dim S$. This contradicts \eqref{EqRBGIsotrivial}. Hence we are done in this case.

\noindent\boxed{\text{Case: General isotrivial abelian scheme}} Now we go back to an arbitrary isotrivial abelian scheme $\cA \rightarrow S$.

There exists a finite and surjective morphism 
$\rho \colon S' \rightarrow S$, with $S'$ irreducible, such that the base change $\cA' :=
\cA\times_S S' \rightarrow S'$ is a trivial abelian scheme. Denote by
$\rho_{\cA} \colon \cA' \rightarrow \cA$ the natural projection.
Then $\rho_{\cA}$ is finite and surjective, so $\dim\cA'=\dim\cA$. Moreover, there is 
 an irreducible component $X'$ of $\rho_{\cA}^{-1}(X)$ with
$\rho_{\cA}(X')=X$ and  $\dim X' = \dim X$.
So $X'$ dominates $S'$ and $X'_{\overline{\eta}}$ is irreducible, as
$X_{\overline{\eta}}$ is. Moreover, $X'_{\overline{\eta}}$ is not
contained in any proper algebraic subgroup of $\cA'_{\overline{\eta}}
= \cA_{\overline{\eta}}$, and $\codim_{\cA'}X' = \codim_{\cA} X > \dim
S = \dim S'$. Finally, $\rho_{\cA}^*\cL$ is relatively ample on $\cA'
\rightarrow S'$.

We have proved the relative Bogomolov conjecture for the trivial
abelian scheme $\cA' \rightarrow S'$. So there exists $\epsilon > 0$
such that
\[
X'(\epsilon; \rho_{\cA}^*\cL) = \{x' \in X'(\IQbar) : \hat{h}_{\rho_{\cA}^*\cL}(x') \le \epsilon\}
\]
is not Zariski dense in $X'$. In particular
\begin{equation}\label{EqRelBogFiniteCover}
\dim \overline{X'(\epsilon;\rho_{\cA}^*\cL)}^{\mathrm{Zar}} < \dim X' = \dim X.
\end{equation}

It is not hard to check $\rho_{\cA}(X'(\epsilon;\rho_{\cA}^*\cL)) =
X(\epsilon; \cL)$ using $ \hat{h}_{\rho_{\cA}^*\cL}(x') =
\hat{h}_{\cL}(\rho_{\cA}(x'))$ and $\rho_{\cA'}(X') = X$. Therefore,
and 
as $\rho_{\cA}$ is a closed morphism, 
\[
\rho_{\cA}(\overline{X'(\epsilon;\rho_{\cA}^*\cL)}^{\mathrm{Zar}}) = \overline{\rho_{\cA}(X'(\epsilon;\rho_{\cA}^*\cL))}^{\mathrm{Zar}}= \overline{X(\epsilon; \cL)}^{\mathrm{Zar}}.
\]
So we have $\dim \overline{X(\epsilon; \cL)}^{\mathrm{Zar}} =
\dim\rho_{\cA}(\overline{X'(\epsilon;\rho_{\cA}^*\cL)}^{\mathrm{Zar}})
= \dim \overline{X'(\epsilon;\rho_{\cA}^*\cL)}^{\mathrm{Zar}}$ because
$\rho_\cA$ is finite. By \eqref{EqRelBogFiniteCover} we then have $\dim \overline{X(\epsilon; \cL)}^{\mathrm{Zar}} < \dim X$. Hence $X(\epsilon;\cL)$ is not Zariski dense in $X$. We are done.
\end{proof}

\bibliographystyle{alpha}
\bibliography{literature}

\end{document}